\documentclass{article}
\usepackage[utf8]{inputenc}
\usepackage{amsfonts}
\usepackage{amssymb}
\usepackage{amsmath}
\usepackage{amsthm}
\usepackage{commath}
\usepackage{xcolor}
\usepackage{bbold}
\usepackage{mathrsfs}
\usepackage{enumitem}
\usepackage{graphicx}
\usepackage{cancel}
\usepackage[noabbrev]{cleveref}
\usepackage{tikz-cd}

\newcommand{\R}{\mathbb{R}}

\newcommand{\T}{\mathbb{T}}

\newcommand{\C}{\mathbb{C}}
\newcommand{\vphi}{\varphi}
\newcommand{\Id}{\textup{Id}}
\newcommand{\A}{\mathcal{A}}

\newcommand{\wt}[1]{\widetilde{#1}}

\newcommand{\grad}{\textup{grad}\,}
\renewcommand{\del}{\partial}

\renewcommand{\L}{\mathcal{L}}

\newcommand{\V}{\mathcal{V}}

\newcommand{\pr}{\text{pr}}
\newcommand{\J}{\mathfrak{J}}

\newtheorem{theorem}{Theorem}
\newtheorem{definition}[theorem]{Definition}
\newtheorem{lemma}[theorem]{Lemma}

\newtheorem{remark}[theorem]{Remark}
\newtheorem{proposition}[theorem]{Proposition}
\newtheorem{example}[theorem]{Example}

\tikzset{every picture/.style={line width=0.75pt}} %set default line width to 0.75pt       

\numberwithin{theorem}{section}
\title{Floer sections in multisymplectic geometry}
\author{Ronen Brilleslijper and Oliver Fabert }
\date{}

\begin{document}

\maketitle

\begin{abstract}
    In symplectic geometry, Floer theory is the most important tool to prove the existence of time-periodic solutions in Hamiltonian mechanics. The core observation is that the $L^2$-gradient lines of the symplectic action functional are pseudo-holomorphic curves, enabling the use of elliptic PDE methods. Multisymplectic geometry is the geometric framework underlying Hamiltonian field theory, where the time line is replaced by higher-dimensional manifolds. In the case of two dimensions and using complex structures, we introduce a novel multisymplectic framework that is fit for the generalization of the elliptic methods from symplectic geometry. Besides proving a Darboux theorem, we show that the $L^2$-gradient lines of our multisymplectic action functional are now pseudo-Fueter curves defined using a compatible almost hyperk\"ahler structure.
\end{abstract}

\tableofcontents

%\section*{Introduction}

In symplectic geometry, one studies solutions $u:S^1\to M$ to Hamilton's equation 
\begin{align}\label{eq:Hamilton}
\omega(\cdot,\dot{u}(t))=dH_t(u(t)) 
\end{align}
on a symplectic manifold $(M,\omega)$ with periodic Hamiltonian $H:S^1\times M\to \R$. One of the most powerful tools to study the existence of these 1-dimensional objects is Floer theory. The idea is that solutions to \cref{eq:Hamilton} can be seen as critical points of the action functional. One can therefore study negative gradient lines of the action functional and use ideas from Morse theory to find solutions to Hamiltons equation. These negative gradient lines are referred to as Floer curves and are given by maps $u:\R\times S^1\to M$ satisfying the Floer equation
\begin{align}\label{eq:symplFloer}
    \del_su+J\del_tu=\nabla H(u),
\end{align}
where $J$ is an almost complex structure compatible with $\omega$ and the gradient is taken with respect to the induced metric $\omega(J\cdot,\cdot)$. Floer theory fundamentally uses the fact that \cref{eq:symplFloer} is an elliptic PDE. 

When moving from Hamiltonian mechanics to field theory, 1-dimensional time is replaced by higher-dimensional manifolds. Multisymplectic geometry is a generalization of symplectic geometry to the setting where we study maps from $d$-dimensional manifolds. In \cref{sec:multisympl} we will dive more into the framework of multisymplectic geometry. For more background on multisymplectic and related geometries see \cite{deLeon,forger2013multisymplectic,marsden1999multisymplectic}. A natural question is whether we can study the existence of these higher-dimensional objects using techniques similar to Floer theory on symplectic manifolds. So far, there is no generalization of Floer theory to the multisymplectic setting. The reason is that one of the most fundamental equations in multisymplectic geometry is not elliptic. These equations, called the De Donder-Weyl equations, locally have the following form
\begin{align}\label{eq:DDW}
\begin{split}
     -\del_1 p^a_1-\del_2 p^a_2&=\frac{\del H}{\del q^a}\\
    \del_1 q^a&=\frac{\del H}{\del p_1^a}\\
   \del_2 q^a&=\frac{\del H}{\del p_2^a},
   \end{split}
\end{align}
for $d=2$ and local coordinates $q^a,p_1^a,p_2^a$ on the multisymplectic manifold. It can easily be seen that the differential operator on the left-hand side of these equations has an infinite-dimensional kernel (see \cite{from1to2}). This makes the use of Floer-theoretic methods impossible. 

If the underlying 2-dimensional manifold is a torus, the multisymplectic framework simplifies drastically to what we call polysymplectic geometry in the sense of \cite{gunther1987polysymplectic}.\footnote{The authors have been made aware that the term polysymplectic nowadays can have a slightly different meaning.} In polysymplectic geometry, the same problem arises as mentioned above, namely that the underlying equations are not elliptic. However, in \cite{from1to2} a formalism was developed in which the existence of 2-dimensional objects can be studied, in analogy with the 1-dimensional objects, that can be studied in symplectic geometry. The resulting theory, which we called complex-regularized polysymplectic (CRPS) geometry, provides a subclass of polysymplectic manifolds on which the underlying equations are in fact elliptic.

In CRPS geometry, one starts with a complex manifold $(W,I)$. A closed $\R^2$-valued 2-form $\Omega=\omega_1\otimes\del_1+\omega_2\otimes\del_2$ on $W$ is called CRPS if $\ker\Omega^\flat=0$ and $\omega_2=-\omega_1(\cdot,I\cdot)$. In \cite{from1to2} it was shown, that this implies that both $\omega_1$ and $\omega_2$ are in fact symplectic forms. The main example of a CRPS manifold, is the cotangent bundle of a complex manifold $Q$ with its induced complex structure $I$. An easy check shows that $\omega_1=d\theta$ and $\omega_2=-d(\theta\circ I)$ yield a CRPS form, where $\theta=pdq$ is the canonical Liouvile 1-form. Locally, all CRPS manifolds are isomorphic.

Given a CRPS manifold $(W,I,\Omega=\omega_1\otimes\del_1+\omega_2\otimes\del_2)$, there is an action functional on the space of maps $\T^2\to W$. This action functional arises from the multisymplectic form $\wt\Omega=\omega_1\wedge dt_2-\omega_2\wedge dt_1$ on $W\times\T^2$, that is obtained by contracting $\Omega$ with the volume form $dt_1\wedge dt_2$ on $\T^2$. Assume that $\omega_j=d\theta_j$. Then we can define 
\begin{align*}
    \A(Z)=\int_{\T^2}\wt{Z}^*\wt{\Theta} = \int_{\T^2}(\theta_1(\del_1Z)+\theta_2(\del_2Z))dt_1\wedge dt_2,
\end{align*}
where $\wt{Z}=Z\times\Id:\T^2\to W\times\T^2$ and $\wt\Theta$ is the primitive of $\wt\Omega$. Given a Hamiltonian $H:W\to\R$, we may also define $\A_H(Z)=\A(Z)-\int_{T^2}H(Z)dt_1\wedge dt_2$. Critical point of $\A_H$ are given by the equation
\begin{align}\label{eq:polysympl}
    \omega_1(\cdot,\del_1Z)+\omega_2(\cdot,\del_2Z)=dH.
\end{align}

Now, as shown in \cite{from1to2}, there exist anti-commuting almost complex structures $J$ and $K$ on $W$, such that $\omega_1=g(\cdot,J\cdot)$ and $\omega_2=g(\cdot,K\cdot)$, for some Riemannian metric $g$. Taking the gradient of $\A_H$ with respect to the $L^2$-metric induced by $g$, yields $\grad\A_H(Z)=J\del_1Z+K\del_2Z-\nabla H(Z)$. In \cite{from1to2}, negative gradient lines of $\A_H$ were used to determine a lower bound on the number of solutions of \cref{eq:polysympl}, using a similar reasoning as in Floer theory. A crucial fact there was that for $H=0$, these negative gradient lines are given by the Fueter equation, that has been studied by \cite{HNS,walpuski2017compactness,salamon2013three}. See also the work by Doan-Rezchikov and Kontsevich-Soibelman in \cite{doanrez,kontsevichsoibel} towards defining holomorphic Floer theory.

For the construction above it is crucial that $\T^2$ is parallelizable, since the $\R^2$-valued form $\Omega$ can then be seen as taking values in the tangent bundle of $\T^2$. In this article, we would like to present a formalism in which we can extend the above discussion to maps from other Riemann surfaces $\Sigma$. As discussed above, we need to leave the world of polysymplectic geometry and use the multisymplectic framework to study maps on different surfaces. The idea is that the multisymplectic 3-form $\wt\Omega$ on $W\times\T^2$ actually generalizes well to other surfaces. In the next section, we will start by reviewing the setup of multisymplectic geometry, before specifying a subclass of manifolds, which we call complex-regularized multisymplectic (CRMS) manifolds, that are suitable for studying Floer curves like explained above. Even though the motivation of this article is to generalize the framework of \cite{from1to2} to surfaces other than $\T^2$, this paper is self-contained so no results from \cite{from1to2} are needed in order to follow the exposition.

This paper is organized as follows. In \cref{sec:multisympl} we give an introduction to multisymplectic geometry by discussing the extended multimomentum bundle that generalizes the cotangent bundle from symplectic geometry. Subsequently, we incorporate the complex structures that are needed to produce elliptic equations in \cref{sec:cotagent}. \Cref{sec:CRPS} relates the construction from \cref{sec:cotagent} to the CRPS setting from \cite{from1to2} in the case where $\Sigma=
\T^2$. Next, we study the structures on the extended multimomentum bundle when $\Sigma\neq\T^2$ in \cref{sec:CRMS} that motivate the general definition of CRMS forms in \cref{sec:generaldef}. A version of the Darboux theorem for CRMS forms is proven in \cref{sec:Darboux}. Finally, in \cref{sec:Floer}, we show that the $L^2$-gradient lines of the multisymplectic action functional on CRMS bundles satisfy the perturbed Fueter equation, which is the analogue of the perturbed Cauchy-Riemann equation from symplectic geometry. This allows for the generalization of the elliptic methods that led to the success of symplectic geometry since the 1980s to multisymplectic geometry, which was the motivation that led to this article.

\section{Multisymplectic geometry}\label{sec:multisympl}
We will start by reviewing the setup of multisymplectic geometry. For clarity of the concepts and to see what the resulting equations look like, we will indicate local coordinates for the objects introduced. This section is based on chapter 15 of \cite{deLeon}.

Fields are defined to be sections of a fiber bundle $\pi:Y\to \Sigma$, where $\Sigma$ is a surface. Eventually, we want to study maps $\Sigma\to Q$, so that we are interested only in the trivial fiber bundle $Y=\Sigma\times Q$, but since the setup of the theory makes sense for general fiber bundles, we will not yet restrict ourselves here. Locally, we say that $\Sigma$ has coordinates $t:=(t_1,t_2)$ and the fiber $Y_t$ of $Y$ over $t$ has coordinates $q=(q^1,\ldots,q^n)$. Let $\Lambda^2(Y)\to Y$ be the bundle of $2$-forms over Y. We can define a subbundle by 
\[
\Lambda^2_2(Y) := \{\lambda\in\Lambda^2(Y)\mid u\lrcorner(v\lrcorner\lambda)=0\text{ for all }u,v\in\ker d\pi\},
\]
called the \textit{extended multimomentum bundle}. We denote $\kappa:\Lambda^2_2(Y)\to Y$ to be the bundle map. Given the local coordinate expressions for $Y$ and $\Sigma$, we get induced coordinates $(t_1,t_2,q^a,p_1^a,p_2^a,p)$ on $\Lambda^2_2(Y)$, such that an element of $\Lambda^2_2(Y)_y$ is given by \[pdt_1\wedge dt_2+\sum_a\left( p_1^a dq^a\wedge dt_2-p_2^a dq^a\wedge dt_1\right),\]
where $y=(t_1,t_2,q^a)$. 

\begin{definition}
    The \emph{tautological form} $\Theta$ on $\Lambda^2_2(Y)$ is defined by 
    \[\Theta_\lambda = \kappa^*\lambda\]
    for $\lambda\in\Lambda^2_2(Y)$. The canonical multisymplectic form on the extended multimomentum bundle is $\Omega=d\Theta$.
\end{definition}

In local coordinates we have that $\Theta$ and $\Omega$  are given by
\begin{align*}
    \Theta&=pdt_1\wedge dt_2+\sum_a\left( p_1^a dq^a\wedge dt_2-p_2^a dq^a\wedge dt_1\right)\\
    \Omega&=dp\wedge dt_1\wedge dt_2+\sum_a\left( dp_1^a\wedge dq^a\wedge dt_2-dp_2^a\wedge dq^a\wedge dt_1\right).
\end{align*}
Note that $\Omega$ is a non-degenerate $3$-form on $\Lambda^2_2(Y)$ in the sense that $X\mapsto X\lrcorner\Omega$ is an injective map on $T_\lambda\Lambda^2_2(Y)$ for every $\lambda\in\Lambda^2_2(Y)$. 

Inside $\Lambda^2_2(Y)$, we have the subbundle $\pi^*\Lambda^2(\Sigma)$ of 2-forms on the base. We define the \emph{restricted multimomentum bundle} to be the quotient \[\wt{\Lambda^2_2(Y)}=\frac{\Lambda^2_2(Y)}{\pi^*\Lambda^2(\Sigma)}\]
and let $\mu:\Lambda^2_2(Y)\to \wt{\Lambda^2_2(Y)}$ be the canonical projection. Local coordinates on $\wt{\Lambda^2_2(Y)}$ are given by $(t_1,t_2,q^a,p_1^a,p_2^a)$. 

\begin{definition}
    A \emph{Hamiltonian section} is a section $h$ of $\mu$. It induces the 2-form $\Theta^h:=h^*\Theta$ and the 3-form $\Omega^h=h^*\Omega=d\Theta^h$ on the restricted multimomentum bundle. \newline
    A section $Z$ of $\sigma:\wt{\Lambda^2_2(Y)}\to\Sigma$ is called a solution of $h$, if
    \begin{align*}
        Z^*(X\lrcorner\Omega^h)=0
    \end{align*}
    for all $X\in\mathfrak{X}(\wt{\Lambda^2_2(Y)})$.
\end{definition}

Locally a Hamiltonian section is given by 
\[h(t_1,t_2,q^a,p_1^a,p_2^a)=(t_1,t_2,q^a,p_1^a,p_2^a,-H(t_1,t_2,q^a,p_1^a,p_2^a)),\]
where $H$ is a smooth local function on $\wt{\Lambda^2_2(Y)}$. Thus, in this local expression 
\[\Omega^h=-dH\wedge dt_1\wedge dt_2+\sum_a\left( dp_1^a\wedge dq^a\wedge dt_2-dp_2^a\wedge dq^a\wedge dt_1\right).\]
Let $Z$ be a section of $\sigma$, locally given by $(t_1,t_2)\mapsto(t_1,t_2,q^a,p_1^a,p_2^a)$. We see that 
\begin{align*}
    Z^*\left(\frac{\del}{\del q^a}\lrcorner\Omega^h\right) &= -\left(\frac{\del H}{\del q^a}+\del_1 p^a_1+\del_2 p^a_2\right)dt_1\wedge dt_2\\
    Z^*\left(\frac{\del}{\del p_1^a}\lrcorner\Omega^h\right) &= -\left(\frac{\del H}{\del p_1^a}-\del_1 q^a\right)dt_1\wedge dt_2\\
    Z^*\left(\frac{\del}{\del p_2^a}\lrcorner\Omega^h\right) &= -\left(\frac{\del H}{\del p_2^a}-\del_2 q^a\right)dt_1\wedge dt_2.
\end{align*}
Here, $\del_1$ and $\del_2$ denote the derivatives with respect to $t_1$ and $t_2$. Therefore, $Z$ is a solution of $h$ if and only if in local coordinates it satisfies the De Donder-Weyl \cref{eq:DDW}. A straightforward check shows that the equations obtained by contracting $\Omega^h$ with $\del_1$ and $\del_2$ merely express the total derivative of $H$ in terms of its partial derivatives as 
\[\frac{d (H\circ Z)}{d t_i} =\frac{\del H}{\del t_i}+ \sum_a\left(\frac{\del H}{\del q^a}\del_i q^a+\frac{\del H}{\del p^a_1}\del_{i} p^a_1 +\frac{\del H}{\del p^a_2}\del_{i}p^a_2\right).\]

Alternatively, when $\Sigma$ is a closed surface, an easy calculation shows that $Z$ is a solution of $h$ if and only if it is a critical point of the action functional $\A_h$ defined by 
\begin{align}\label{eq:action}
\A_h(Z) = \int_\Sigma Z^*\Theta^h.
\end{align}

This action functional is not suitable to study Floer theory, for similar reasons as in the polysymplectic framework (see \cite{from1to2}). Thus, we restrict ourselves to a smaller class.

\section{Incorporating complex structures}\label{sec:cotagent}
Instead of starting with a general fiber bundle, like in \cref{sec:multisympl}, we now consider a holomorphic fiber bundle $\pi:(Y,i)\to(\Sigma,j)$. I.e. $(Y,i)$ is a complex manifold, $(\Sigma, j)$ a Riemann surface, the bundle map satisfies $d\pi\circ i=j\circ d\pi$ and locally $Y$ is biholomorphic to the trivial bundle.\footnote{Again, we are eventually only interested in the trivial bundle, but provide a more general setup of the theory.} 

Using the complex structure $i$, we may define a subbundle $\Lambda^C(Y)$ of the extended multimomentum bundle by
\[\Lambda^C(Y) := \{\lambda\in\Lambda^2_2(Y)\mid \lambda(i\cdot,i\cdot)=\lambda\}.\] 
Note that this definition makes sense, since $\lambda$ is an element of $\Lambda^2_2(Y)$ if and only if $\lambda(i\cdot,i\cdot)$ is an element of $\Lambda_2^2(Y)$.
The tautological form on $\Lambda^2_2(Y)$ restricts to a 2-form on $\Lambda^C(Y)$, which we still denote by $\Theta$. The 3-form $\Omega=d\Theta$ on $\Lambda^C(Y)$ is called the \emph{complex-regularized multisymplectic form} on $\Lambda^C(Y)$. 

We note that $\pi^*\Lambda^2(\Sigma)\subseteq\Lambda^C(Y)$, since $j$ preserves 2-dimensional volume (as can be checked in local coordinates). Therefore, we may still define the quotient
\[\wt{\Lambda^C(Y)}:=\frac{\Lambda^C(Y)}{\pi^*\Lambda^2(\Sigma)}.\]
The projection is denoted by $\mu^C:\Lambda^C(Y)\to\wt{\Lambda^C(Y)}$. 

\begin{definition}
    A \emph{complex-regularized Hamiltonian section} is defined to be a section $h$ of $\mu^C$. As before it induces the forms $\Theta^h=h^*\Theta$ and $\Omega^h=h^*\Omega$. \newline
    A section $Z$ of $\sigma^C:\wt{\Lambda^C(Y)}\to\Sigma$ is called a solution of $h$ if
    \begin{align*}
        Z^*(X\lrcorner\Omega^h)=0
    \end{align*}
    for all $X\in\mathfrak{X}(\wt{\Lambda^C(Y)})$.
\end{definition}

We want to show that $Z$ is a solution of $h$ if and only if it locally satisfies the Bridges equations from \cite{from1to2}. We may choose local coordinates $(t_1,t_2)$ on $\Sigma$ and $(t_1,t_2,q^a_1,q^a_2)$ on $Y$, such that
\begin{align*}
    j\del_1=\del_2 && i\frac{\del}{\del q_1^a}=\frac{\del}{\del q_2^a}.
\end{align*}
They induce local coordinates $(t_1,t_2,q^a_1,q^a_2,p^a_{1,1},p^a_{1,2},p^a_{2,1},p^a_{2,2},p)$ on $\Lambda^2_2(Y)$ such that an element in $\Lambda^2_2(Y)_y$ is given by
\[\lambda=pdt_1\wedge dt_2+\sum_a\left(p^a_{1,1}dq^a_1\wedge dt_2-p^a_{1,2}dq^a_1\wedge dt_1+p^a_{2,1}dq^a_2\wedge dt_2-p^a_{2,2}dq^a_2\wedge dt_1\right)\]
for $y=(t_1,t_2,q^a_1,q^a_2)$. In these coordinates, we get that 
\begin{align*}
    \lambda(i\cdot,i\cdot) = pdt_1\wedge dt_2+\sum_a\left(p^a_{2,2}dq^a_1\wedge dt_2+p^a_{2,1}dq^a_1\wedge dt_1-p^a_{1,2}dq^a_2\wedge dt_2-p^a_{1,1}dq^a_2\wedge dt_1\right).
\end{align*}
Thus, $\lambda$ is an element of $\Lambda^C(Y)$ if and only if
\begin{align*}
    p^a_{1,1}=p^a_{2,2} && p^a_{1,2}=-p^a_{2,1}
\end{align*}
for all $a$. We may define $P^a_1:=p^a_{1,1}=p^a_{2,2}$ and $P^a_2:=-p^a_{1,2}=p^a_{2,1}$. Then $\Lambda^C(Y)$ has coordinates $(t_1,t_2,q^a_1,q^a_2,P^a_1,P^a_2,p)$, where elements in the fiber are given by 
\[\lambda=pdt_1\wedge dt_2 +\sum_a(P^a_1dq^a_1+P^a_2dq^a_2)\wedge dt_2-\sum_a(P^a_1dq^a_2-P^a_2dq^a_1)\wedge dt_1.\]
Therefore, we compute $\Omega$ as
\[\Omega=dp\wedge dt_1\wedge dt_2 +\sum_a(dP^a_1\wedge dq^a_1+dP^a_2\wedge dq^a_2)\wedge dt_2-\sum_a(dP^a_1\wedge dq^a_2-dP^a_2\wedge dq^a_1)\wedge dt_1,\]
which is indeed pointwise non-degenerate.

Now $(t_1,t_2,q^a_1,q^a_2,P^a_1,P^a_2)$ define local coordinates on $\wt{\Lambda^C(Y)}$ and a Hamiltonian section is locally given by $p=-H(t_1,t_2,q^a_1,q^a_2,P^a_1,P^a_2)$ for some function $H$ as above. 
\begin{proposition}
    Let $Z:\Sigma\to\wt{\Lambda^C(Y)}$ be a section that is given by $(t_1,t_2)\mapsto (t_1,t_2,q^a_1,q^a_2,P^a_1,P^a_2)$ in the local coordinates defined above. Then $Z$ is a solution of the Hamiltonian section $h$ given locally by $p=-H(t_1,t_2,q^a_1,q^a_2,P^a_1,P^a_2)$ if and only if in local coordinates $Z$ satisfies
    \begin{align}
    \begin{split}\label{eq:Bridges}
        \frac{\del H}{\del q^a_1} &= -\del_1 P^a_1 + \del_2 P^a_2\\
        \frac{\del H}{\del q^a_2} &= -\del_1 P^a_2 - \del_2 P^a_1\\
        \frac{\del H}{\del P^a_1} &= \phantom{-}\del_1 q^a_1 + \del_2 q^a_2\\
        \frac{\del H}{\del P^a_2} &= \phantom{-}\del_1 q^a_2 - \del_2 q^a_1.
    \end{split}
    \end{align}
\end{proposition}
\begin{proof}
    In the coordinates indicated above, we see that 
    \[\Omega^h = -dH\wedge dt_1\wedge dt_2 +\sum_a(dP^a_1\wedge dq^a_1+dP^a_2\wedge dq^a_2)\wedge dt_2-\sum_a(dP^a_1\wedge dq^a_2-dP^a_2\wedge dq^a_1)\wedge dt_1.\]
    Thus contracting $\Omega^h$ with different basis vectors gives
    \begin{align*}
        Z^*\left(\frac{\del}{\del q^a_1}\lrcorner\Omega^h\right) &= -\left(\frac{\del H}{\del q^a_1}+\del_1 P^a_1-\del_2 P^a_2\right)dt_1\wedge dt_2\\
        Z^*\left(\frac{\del}{\del q^a_2}\lrcorner \Omega^h\right)&= -\left(\frac{\del H}{\del q^a_2}+\del_1 P^a_2+\del_2 P^a_1\right)dt_1\wedge dt_2\\
        Z^*\left(\frac{\del}{\del P^a_1}\lrcorner\Omega^h\right)&= -\left(\frac{\del H}{\del P^a_1}-\del_1q^a_1-\del_2 q^a_2\right)dt_1\wedge dt_2\\
        Z^*\left(\frac{\del}{\del P^a_2}\lrcorner\Omega^h\right)&= -\left(\frac{\del H}{\del P^a_2}-\del_1q^a_2+\del_2 q^a_1\right)dt_1\wedge dt_2.
    \end{align*}
    Indeed, we see that $Z$ is a solution precisely when \cref{eq:Bridges} is satisfied. Just like before, the contraction of $\Omega^h$ with $\del_1$ and $\del_2$ doesn't give any additional information, but merely expresses the chain rule for the derivatives of $H\circ Z$.
\end{proof}

\begin{example}
   For $Y=\Sigma\times \C^n \to \Sigma$ the trivial bundle, we want to show that we recover the Laplace equation from \cref{eq:Bridges}, just like in \cite{from1to2}. We see that $\mu^C:\Lambda^C(Y)\to\wt{\Lambda^C(Y)}$ is trivial, since a global section is given by a volume form on $\Sigma$. Locally, the volume form is given by $\rho(t)dt_1\wedge dt_2$, for some positive real-valued function $\rho$ on $\Sigma$. Let $H$ be a locally defined function on $\wt{\Lambda^C(Y)}$, given by $(t_1,t_2,q_1^a,q_2^a,P_1^a,P_2^a)\mapsto \frac{1}{2}|P|^2+\lambda(t)V(q)$. This yields a local section $h:Z\mapsto Z-H(Z)dt_1\wedge dt_2$ of $\mu^C$. We show that this actually defines a global section. Note first that $V(q)$ is globally well-defined and so is $\lambda(t)dt_1\wedge dt_2$, since this is the local expression of the volume form. Thus $Z\mapsto \lambda(t)V(q)dt_1\wedge dt_2$ yields a global section. Let $w=\psi(t)$ be a different holomorphic coordinate on $\Sigma$. Then $dw_1\wedge dw_2=|\del_t\psi|^2dt_1\wedge dt_2$. Denote by $(w,q,S)$ the corresponding coordinates on $\wt{\Lambda^C(Y)}$. Then a straightforward calculation shows that $|S|^2=|P|^2|\del_t\psi|^{-2}$. Thus $|S|^2dw_1\wedge dw_2=|P|^2dt_1\wedge dt_2$ and we see that $h$ defines a global section of $\mu^C$. By filling in the third and fourth equalities of \cref{eq:Bridges} into the first and second, we find that solutions to this Hamiltonian section are given by 
   \begin{align*}
       -\frac{4}{\rho(t)}\del_t\del_{\bar{t}}q^a_j=\frac{\del V}{\del q^a_j}.
   \end{align*}
   Note that $\frac{4}{\rho(t)}\del_t\del_{\bar{t}}$ is the Laplace-Beltrami operator on $\Sigma$.
\end{example}

\section{Complex-regularized polysymplectic geometry}\label{sec:CRPS}
Recall that the main example in CRPS geometry was the cotangent bundle of a complex manifold $(Q,i')$. The tautological 1-form on $T^*Q$ is defined by $(\theta_1)_\nu=\nu(\pr_*\cdot)$ for $\nu\in T^*Q$ and $\pr:T^*Q\to Q$ the projection. Using the complex structure, we can naturally define a second 1-form given by $(\theta_2)_\nu=\nu(i'\circ\pr_*\cdot)$. Recall that the CRPS forms are defined by $\omega_1=d\theta_1$ and $\omega_2=-d\theta_2$. We want to show that when $\Sigma=\C$ and $Y=\C\times Q$ is the trivial bundle with the product complex structure, the CRMS construction above recovers these two forms on $T^*Q$.

Let $\pi:Y=\C\times Q\to \C$ be the trivial bundle. We denote by $j$ the standard complex structure on $\C$ and by $i=j\oplus i'$ the product complex structure on $Y$. Also denote by $\pi_Q:Y\to Q$ the canonical projection and by $\iota_t:Q\to Y$ the map $q\mapsto(t,q)$. As is proven in \cite{deLeon}, there is a diffeomorphism $\Psi:\Lambda^2_2(Y)\to \C\times\R\times \left(T^*Q\oplus T^*Q\right)$, given by \[\Psi(\lambda_{(t,q)})=(t,\lambda_{(t,q)}(\del_1,\del_2),\nu_1(\lambda)_q,\nu_2(\lambda)_q),\] where
\begin{align*}
    \nu_1(\lambda)_q(V) &= \lambda_{(t,q)}((\iota_t)_*V,\del_2)\\
    \nu_2(\lambda)_q(V)&=\lambda_{(t,q)}(\del_1,(\iota_t)_*V),
\end{align*}
for $V\in T_qQ$. The inverse is given by
\[\Psi^{-1}(t,p,\nu_1,\nu_2)=pdt_1\wedge dt_2+(\pi_Q)^*\nu_1\wedge dt_2 -(\pi_Q)^*\nu_2\wedge dt_1.\]

We can restrict $\Psi$ to $\Lambda^C(Y)\subseteq\Lambda^2_2(Y)$. For $\lambda\in\Lambda^2_2(Y)$, denote $\lambda^i=\lambda(i\cdot,i\cdot)$. Given $\lambda\in\Lambda^C(Y)$, we know $\Psi(\lambda)=\Psi(\lambda^i)$, which yields
\begin{align*}
    \nu_1(\lambda)(V)=\nu_1(\lambda^i)(V)=\lambda(\del_1,i(\iota_t)_*V)=\lambda(\del_1,(\iota_t)_*(i'V))=\nu_2(\lambda)(i'V),
\end{align*}
so that $\nu_1(\lambda)=\nu_2(\lambda)(i'\cdot)$ or equivalently $\nu_2(\lambda)=-\nu_1(\lambda)(i'\cdot)$. Vice versa, given $\nu_2=-\nu_1(i'\cdot)\in T^*Q$, we get
\begin{align*}
    \Psi^{-1}(t,p,\nu_1,\nu_2)(i\cdot,i\cdot) &= pdt_1\wedge dt_2+ (\pi_Q)^*(\nu_1(i'\cdot))\wedge dt_1+(\pi_Q)^*(\nu_2(i'\cdot))\wedge dt_2\\
    &=pdt_1\wedge dt_2 - (\pi_Q)^*\nu_2\wedge dt_1 + (\pi_Q)^*\nu_1\wedge dt_2\\
    &=\Psi^{-1}(t,p,\nu_1,\nu_2).
\end{align*}
We conclude that the image of $\Psi$ restricted to $\Lambda^C(Y)$ is $\C\times\R\times\{(\nu_1,\nu_2)\in T^*Q\oplus T^*Q\mid \nu_2=-\nu_1(i'\cdot)\}$. Note that the last term is diffeomorphic to $T^*Q$. We get a diffeomorphism $\Psi^C:\Lambda^C(Y)\to\C\times\R\times T^*Q$ given by
\begin{align*}
    \Psi^C(\lambda_{(t,q)})&=(t,\lambda_{(t,q)}(\del_1,\del_2),\nu_1(\lambda))\\&=(t,\lambda_{(t,q)}(\del_1,\del_2),\nu_2(\lambda)(i'\cdot))\\
    (\Psi^C)^{-1}(t,p,\nu) &= pdt_1\wedge dt_2+(\pi_Q)^*\nu\wedge dt_2+(\pi_Q)^*(\nu(i'\cdot))\wedge dt_1
\end{align*}
Using this diffeomorphism, we can transfer the 2-form $\Theta$ from $\Lambda^C(Y)$ to $\C\times\R\times T^*Q$. We also notice that $\Lambda^C(Y)\to\wt{\Lambda^C(Y)}\cong \C\times T^*Q$ is now the trivial bundle, so we can pullback $\Theta$ to $\wt{\Lambda^C(Y)}$ along the zero section. This yields a 2-form $\Theta^0$ on $\C\times T^*Q$. Using the formulas for $\Psi^C$ above, a straightforward but lengthy diagram chase shows that 
\begin{align*}
    \del_1\lrcorner(\del_2\lrcorner\Theta^0)&=0\\
    (\tau_t)^*(\del_1\lrcorner\Theta^0)&=-\theta_2\\
    (\tau_t)^*(\del_2\lrcorner\Theta^0)&=-\theta_1,
\end{align*}
where $\tau_t:T^*Q\to\C\times T^*Q$ is the map $\nu\mapsto(t,\nu)$. Thus, $d\Theta^0=\omega_1\wedge dt_2-\omega_2\wedge dt_1$ on $\C\times T^*Q$, which can be identified with the contraction of the polysymplectic form $\omega_1\otimes\del_1+\omega_2\otimes\del_2$ with the volume form $dt_1\wedge dt_2$ on $\C$. 

Note that the form $\Theta^0$ is precisely the 2-form that was integrated in \cite{from1to2} to define the action functional. Thus the action functional for sections of $\wt{\Lambda^C(Y)}\to\Sigma$ defined in \cref{eq:action} coincides with the action functional in the CRPS framework when $Y=\C\times Q$ is the trivial bundle over $\C$. Since $\Lambda^C(Y)\to\wt{\Lambda^C(Y)}$ is now a trivial bundle, Hamiltonian sections $h$ correspond globally to Hamiltonian functions $H$ and $\Omega^h=-dH\wedge dt_1\wedge dt_2+\omega_1\wedge dt_2-\omega_2\wedge dt_1$. One can compare this to the 2-form $-dH\wedge dt+\omega$ on the mapping torus of a symplectic manifold. 

\section{Complex-regularized multisymplectic geometry}\label{sec:CRMS}
% Recall that in CRPS geometry, the phase space $T^*Q$ comes equipped with an integrable complex structure and there exists a metric such that the forms $\omega_1$ and $\omega_2$ induce almost complex structures $J$ and $K$, such that $IJ=K$. In this section we will construct the analogues in the CRMS setting, by working in local coordinates and proving that the structures are globally well-defined. 

Recall that we are mainly interested in the case where $Y=\Sigma\times Q$ is the trivial bundle, since we want to study maps $\Sigma\to Q$. In this case, we want to show some properties of $\wt{\Lambda^C(Y)}$ that will lead us to the general definition of a CRMS form on a bundle. We start by proving that $\wt{\Lambda^C(Y)}$ is a complex manifold.

Let $(\Sigma, j)$ a Riemann surface, $(Q,i')$ a complex manifold, and define $Y=\Sigma\times Q$ with the product complex structure. We cover $\Sigma$ with charts $\psi_\alpha:U_\alpha\to\psi_
\alpha(U_\alpha)\subseteq\C$, such that all transition maps $\psi_{\alpha\beta}:=\psi_\beta\circ\psi_\alpha^{-1}:\psi_\alpha(U_\alpha\cap U_\beta)\to\psi_\beta(U_\alpha\cap U_\beta)$ are biholomorphic. This means that if $t=t_1+i t_2$ denotes a local coordinate on $\psi_\alpha(U_\alpha)$, that $\bar\del_t\psi_{\alpha\beta}=0$, where $\bar\del_t=\frac{1}{2}(\del_{t_1}+i\del_{t_2})$ is the antiholomorphic derivative. Also, locally $j$ is given by $j\del_{t_1}=\del_{t_2}$.
Similarly, we cover $Q$ with charts $\vphi_\gamma:V_\gamma\to\vphi_\gamma(V_\gamma)\subseteq \C^n$, such that transition maps are biholomorphic. This induces product charts $\rho_{\alpha,\gamma}=\psi_\alpha\times \vphi_\gamma$ on $Y$. 

Pick an auxiliary metric $h$ on $\Sigma$ within the conformal class determined by $j$. This means that there are positive functions $\rho_\alpha$ such that on $\psi_\alpha(U_\alpha)$ the pull-back of the metric is given by 
\[h_\alpha:=(\psi_\alpha^{-1})^*h=\rho_\alpha(t)dt\otimes d\bar{t}=\rho_\alpha(t)(dt_1^2+dt_2^2).\]
This means that on overlapping charts $\rho_\alpha(t)=\rho_\beta(\psi_{\alpha\beta}(t))|\del_t\psi_{\alpha\beta}(t)|^2$. The induced volume form $dV$ on $\Sigma$ is locally given by $\rho_\alpha(t)dt_1\wedge dt_2=\frac{i}{2}\rho_\alpha(t)dt\wedge d\bar{t}$.

\begin{proposition}\label{prop:complexstructure}
    The manifold $\wt{\Lambda^C}(Y)$ carries the structure of a complex manifold, such that the projection $\sigma^C:\wt{\Lambda^C}(Y)\to\Sigma$ is holomorphic. 
\end{proposition}
\begin{proof}
The charts of $Y$ induce charts $\wt{\Lambda^C}(U_\alpha\times V_\gamma)\to\wt{\Lambda^C}(\rho_{\alpha,\gamma}(U_\alpha\times V_\gamma))\subseteq \C\times\C^n\times (\C^n)^*$ for $\wt{\Lambda^C}(Y)$ whose elements are given by
\[\sum_a\left((P_1^adq_1^a+P_2^adq_2^a)\wedge\rho_\alpha(t)dt_2-(P_1^adq_2^a-P_2^adq_1^a)\wedge\rho_\alpha(t)dt_1\right)\]
with respect to local coordinates $t=t_1+it_2$ on $\Sigma$ and $q^a=q^a_1+iq^a_2$ on $Q$. To determine whether $\wt{\Lambda^C}(Y)$ is a complex manifold, we need to calculate the transition functions. 

Let $w=w_1+iw_2$ denote a local coordinate on $\Sigma$ in the chart $U_\beta$ and $r^a=r^a_1+ir^a_2$ in the chart $V_\delta$ of $Q$. We denote elements of $\wt{\Lambda^C}(U_\beta\times V_\delta)$ by 
\[\sum_a\left((S_1^adr_1^a+S_2^adr_2^a)\wedge\rho_\beta(w)dw_2-(S_1^adr_2^a-S_2^adr_1^a)\wedge\rho_\beta(w)dw_1\right).\]

% Define $\psi$ to be the real part of $\psi_{\alpha\beta}$ and $\vphi^a$ to be the real part of $\vphi_{\gamma\delta}^a$. Then using the fact that $\psi_{\alpha\beta}$ and $\vphi_{\gamma\delta}$ are holomorphic, we get that
% \begin{align*}
%     (\psi_{\alpha\beta}\times\vphi_{\gamma\delta})^*\eta &= \frac{\lambda_\alpha(z)}{|\del_z\psi_{\alpha\beta}|^2}\sum_a\left( dq_1^a\wedge dz_1((S_1\del_{q_2}\vphi+S_2\del_{q_1}\vphi)\del_{z_1}\psi-(S_1\del_{q_1}\vphi-S_2\del_{q_2}\vphi)\right)\\
% \end{align*}

A straightforward computation, using the fact that $\psi_{\alpha\beta}$ and $\vphi_{\gamma\delta}$ are holomorphic, yields that
\[S_1^a-iS_2^a = \frac{\del_t\psi_{\alpha\beta}}{\del_{q^a}\vphi_{\gamma\delta}}(P_1^a-iP_2^a).\]
As this is a holomorphic map, we see that the transition function is a holomorphic map between subsets of $\C\times\C^n\times(\C^n)^*$, where we recall that $i^*=-i$. This proves that $\wt{\Lambda^C}(Y)$ is a complex manifold. 
\end{proof}

For the bundle $\sigma^C:\wt{\Lambda^C(Y)}\to\Sigma$, we denote $V^C=\ker d\sigma^C$ the vertical bundle.
\begin{proposition}
    For any complex-regularized Hamiltonian section $h$, the resulting CRMS form $\Omega^h$ on $\wt{\Lambda^C(Y)}$ satisfies
    \begin{enumerate}[label=\roman*)]
        \item $d\Omega^h=0$,
        \item $\Omega^h(v_1,v_2,v_3)=0$ for any $v_1,v_2,v_3\in V^C_\lambda$ and $\lambda\in\wt{\Lambda^C(Y)}$,
        \item for any $\xi\in T_\lambda\wt{\Lambda^C(Y)}$ with $d\sigma^C(\xi)\neq 0$, we have that $\Omega^h(\xi,\cdot,\cdot)$ restricts to a non-degenerate bilinear form on $V^C_\lambda$,
        \item for any $\xi\in T_\lambda\wt{\Lambda^C(Y)}$ and $v_1,v_2\in V^C_\lambda$ we have that 
        \[\Omega^h(I\xi,v_1,v_2)=-\Omega^h(\xi,v_1,Iv_2),\]
        where $I$ is the complex structure on $\wt{\Lambda^C(Y)}$.
    \end{enumerate}
\end{proposition}
\begin{proof}
    The first property follows trivially from $\Omega^h=d\Theta^h$. The other properties can be easily checked in local coordinates. Note that in the local coordinates introduced in \cref{prop:complexstructure}, $V^C_\lambda$ is spanned by $\{\del_{q_1^a},\del_{q_2^a}\del_{P_1^a},\del_{P_2^a}\}$, where $I\del_{q_1^a}=\del_{q_2^a}$ and $I\del_{P_1^a}=-\del_{P_2^a}$. Take $\xi\in T_\lambda\wt{\Lambda^C(Y)}$ given by 
    \[\xi=\frac{\xi_{t_1}}{\rho_\alpha(t)}\del_{t_1}+\frac{\xi_{t_2}}{\rho_\alpha(t)}\del_{t_2}+\sum_a\left(\xi_{q_1^a}\del_{q_1^a}+\xi_{q_2^a}\del_{q_2^a}+\xi_{P_1^a}\del_{P_1^a}+\xi_{P_2^a}\del_{P_2^a}\right),\]
    and define $\omega^\xi:=\Omega^h(\xi,\cdot,\cdot)|_{V^C_\lambda}$ as a 2-form on $V^C_\lambda$.
    Then 
    \begin{align*}
        \omega^\xi=\xi_{t_2}\sum_a\left(dP_1^a\wedge dq_1^a+dP_2^a\wedge dq_2^a\right)-\xi_{t_1}\sum_a\left(dP_1^a\wedge dq_2^a-dP_2^a\wedge dq_1^a\right).
    \end{align*}
    First of all, if $\xi\in V^C_\lambda$, then $\xi_{t_1}=\xi_{t_2}=0$, which proves ii). We write $\omega^\xi=\xi_{t_2}\omega_1-\xi_{t_1}\omega_2$, for 2-forms $\omega_1$ and $\omega_2$ on $V^C_\lambda$. Then an easy check shows $\omega_2=-\omega_1(\cdot,I\cdot)$. Combined with $\omega^{I\xi}=\xi_{t_1}\omega_1+\xi_{t_2}\omega_2$, this gives iv). Also, this implies 
    \[\omega^\xi=\omega_1(\cdot,(\xi_{t_2}\Id+\xi_{t_1}I)\cdot),\]
    and since $\xi_{t_2}\Id+\xi_{t_1}I$ is a non-degenerate matrix whenever $(\xi_{t_1},\xi_{t_2})\neq(0,0)$, we see that indeed $\omega^\xi$ is non-degenerate on $V^C_\lambda$ whenever $d\sigma^C(\xi)\neq 0$.
\end{proof}

\begin{remark}
    We note that $\sigma^C:\wt{\Lambda^C(Y)}\to\Sigma$ is in general not a trivial bundle, even when $\pi:Y\to\Sigma$ is trivial. For example, let $Q=\C$. Let 
    \begin{align*}
        \Phi:\wt{\Lambda^2_2(Y)}&\to Q\times(T^*\Sigma)^{\oplus2}\\
        (t,q,\lambda)&\mapsto \left(q,t,\lambda\left(\del_{q_1},\cdot\right),\lambda\left(\del_{q_2},\cdot\right)\right)
    \end{align*}
    with inverse $(q,t,\eta_1,\eta_2)\mapsto(t,q,\eta_1\wedge dq_1+\eta_2\wedge dq_2)$. This is a bundle isomorphism. It restricts to an isomorphism $\wt{\Lambda^C(Y)}\overset{\sim}{\to} Q\times T^*\Sigma$, where $T^*\Sigma$ sits in $(T^*\Sigma)^{\oplus2}$ by $(t,\eta)\mapsto (t,\eta\circ j,\eta)$. So $\wt{\Lambda^C(Y)}\to\Sigma$ is isomorphic to $\C\times T^*\Sigma\to\Sigma$, which is a non-trivial bundle if $\Sigma\neq\T^2$.
\end{remark}

\section{A general definition}\label{sec:generaldef}
In this paragraph, we will define the concept of \emph{complex-regularized multisymplectic} (CRMS) bundles, inspired by the discussion in the previous paragraphs. The spaces $\wt{\Lambda^C(Y)}$ for $Y=\Sigma\times Q$ will form the main examples of CRMS bundles and we will prove a Darboux-type theorem. 

For a fiber bundle $\sigma:W\to\Sigma$, we denote $V=\ker d\sigma\subseteq TW$ the vertical bundle and $V_w$ the fiber over $w\in W$. 
\begin{definition}\label{def:CRMSbundle}
    Let $\sigma:W\to\Sigma$ be a fiber bundle, where $(W,I)$ is a complex manifold, $(\Sigma,j)$ is a Riemann surface, and $\sigma$ is holomorphic. A 3-form $\Omega$ on $W$ is called \emph{CRMS} if it satisfies the following conditions.
    \begin{enumerate}[label=\roman*)]
        \item $d\Omega=0$;
        \item\label{1hor} $\Omega$ is 1-horizontal, meaning that $\Omega(v_1,v_2,v_3)=0$ for any $v_1,v_2,v_3\in V_w$, for all $w$;
        \item\label{fiberwisenondeg} $\Omega$ is fiberwise non-degenerate, meaning that  for any $\xi\in T_wW$ with $d\sigma(\xi)\neq 0$, we have that $\Omega(\xi,\cdot,\cdot)$ is a non-degenerate bilinear form when restricted to $V_w$;
        \item\label{Icomp} $\Omega$ is $I$-compatible, in the sense that 
        \[\Omega(I\xi, v_1,v_2)=-\Omega(\xi,v_1,I v_2)\]
        for all $\xi\in T_wW$ and $v_1,v_2\in V_w$.
    \end{enumerate}
\end{definition}
Note that CRMS forms are in particular multisymplectic, in the sense that they are closed and 1-nondegenerate (see \cite{deLeon}).

\begin{remark}\label{rem:CRMSbundle}
    \begin{enumerate}[label=\alph*.]
        \item\label{holomfiberbundle} We do not require $W\to\Sigma$ to be a holomorphic fiber bundle. That is, even locally, the complex structure on $W$ does not  have to be the product of the complex structures on $\Sigma$ and the fiber. Instead, in a local trivialization $S\times F$, where $S\subseteq \Sigma$ and $F$ the fiber, $I$ looks like 
        \begin{align*}
            I = \begin{pmatrix}
                j & 0\\
                A & I'
            \end{pmatrix},
        \end{align*}
        for a complex structure $I'$ on $F$ and a map $A:TS\to TF$ such that $A\circ j = I'\circ A$. 
        \item\label{vertindep} For $\xi,\xi'\in T_wW$ with $d\sigma(\xi)=d\sigma(\xi')$, we have $\Omega(\xi,\cdot,\cdot)|_{V_w}=\Omega(\xi',\cdot,\cdot)|_{V_w}$, by condition \ref{1hor}. Thus, the $I$-compatibility condition does not restrict the part of $I$ denoted by $A$ in remark \ref{holomfiberbundle}
    \end{enumerate}
\end{remark}

We will now proceed to formulate a Darboux-type theorem for CRMS bundles. Note that we cannot expect a full Darboux theorem. The definition states that $\Omega$ is fiberwise non-degenerate, but does not restrict the action of $\Omega$ on only one vertical vector. For example, the forms $\Omega^h$ on $\wt{\Lambda^C(Y)}$ defined in \cref{sec:cotagent}, locally look like
\[\Omega^h=\omega_1\wedge dt_2-\omega_2\wedge dt_1-dH\wedge dt_1\wedge dt_2,\]
where $\omega_1$ and $\omega_2$ are the standard CRPS forms on the fibers. Even locally, we cannot get rid of the $dH\wedge dt_1\wedge dt_2$ term. What is true however, is that the action of $\Omega$ on at least two vertical vectors is locally standard. This is the best we can hope for. One should compare this situation to the mapping tori from symplectic geometry. They are described by bundles $M\to S^1$, equipped with a 2-form $\omega_M=\omega-dH\wedge dt$, where $\omega$ is symplectic on the fibers. Also here, we cannot hope to get local coordinates in which $\omega_M$ is standard, but we can however find local coordinates in which the action of $\omega_M$ on the fibers is standard. 

\begin{theorem}\label{thm:Darboux}
    Let $\sigma:W\to \Sigma$ as in \cref{def:CRMSbundle} with CRMS form $\Omega\in\Omega^3(W)$. Then around every point in $W$ we can find an open $\mathbb{U}\subseteq W$ and diffeomorphisms $\psi:\mathbb{U}\to S\times U$ and $\chi=(t_1,t_2): \mathbb{S}=\sigma(\mathbb{U})\to S$, where $S\subseteq \C$ and $U\subseteq \C^{2n}$ open, such that
    \begin{itemize}
        \item $\pr\circ\psi=\chi\circ\sigma$, where $\pr:\C\times\C^{2n}\to\C$ the projection,
        \item $(\psi^{-1})^*\Omega=\Omega_{std}-dH_t\wedge dt_1\wedge dt_2$ for some function $H:S\times U\to \R$,
        \item $\chi$ is holomorphic,
        \item $d\psi\circ I|_V = i\circ d\psi|_V$, for $i$ the standard complex structure on $\C\times\C^{2n}$.
    \end{itemize}
    Here, $\Omega_{std}=\omega_1\wedge dt_2-\omega_2\wedge dt_1$, where $\omega_1$ and $\omega_2$ are the standard CRPS forms on $\C^{2n}$. 
\end{theorem}

%\textcolor{red}{Insert connection to holomorphic symplectic geometry here}

% \begin{remark}
%     The Darboux theorem for CRPS structures, follows as a special case of \cref{thm:Darboux}. Namely, if $\pi:W=M\times\T^2\to\T^2$ and $\omega_1$, $\omega_2$ define a CRPS structure on $M$, then $\Omega=\omega_1\wedge dt_2-\omega_2\wedge dt_1$ defines a CRMS form on $W$. \Cref{thm:Darboux} then implies that $(\psi^{-1})^*\omega_1=(\psi^{-1})^*\Omega(d\psi(\del_2),\cdot,\cdot)|_V=$
% \end{remark}

Before we prove \cref{thm:Darboux}, we start with a linear version of the Darboux theorem.
\begin{lemma}\label{lem:linDarboux}
    Given is a short exact sequence of complex vector spaces
    \[0\to V\to (W,I)\overset{\sigma}{\to} (T,j)\to 0,\]
    where $\dim_\C T=1$. Let $\Omega$ an alternating multilinear 3-form on $W$ that satisfies
    \begin{enumerate}[label=\roman*)]
    \item $\Omega(v_1,v_2,v_3)=0$ for all $v_1,v_2,v_3\in V$,
        \item $\Omega(\xi,\cdot,\cdot)$ is non-degenerate on $V$ for all $\xi\in W$ with $\sigma(\xi)\neq 0$,
        \item $\Omega(I\xi,v_1,v_2)=-\Omega(\xi,v_1,I v_2)$ for all $\xi\in W$ and $v_1,v_2\in V$.
    \end{enumerate}
    Then there exists a real basis $\{e_1,e_2,a_1^k,a_2^k,b_1^k,b_2^k\}_{k=1}^n$ of $W$, such that 
    \begin{itemize}
        \item $T=\text{span}_\R \{\sigma(e_1),\sigma(e_2)\}$ and $V=\text{span}_\R\{a_1^k,a_2^k,b_1^k,b_2^k\}_{k=1}^n$,
        \item $Ie_1=e_2$, $Ia_1^k=a_2^k$, $Ib_1^k=-b_2^k$,
        \item for $\{\epsilon_1,\epsilon_2,\alpha_1^k,\alpha_2^k,\beta_1^k,\beta_2^k\}$ the dual basis, we have
        \[\Omega = \sum_{k=1}^n \left((\beta_1^k\wedge \alpha_1^k+\beta_2^k\wedge \alpha_2^k)\wedge \epsilon_2-(\beta_1^k\wedge\alpha_2^k-\beta_2^k\wedge \alpha_1^k)\wedge\epsilon_1\right)+\nu\wedge\epsilon_1\wedge\epsilon_2,\]
        for some 1-form $\nu$ on $V$. 
    \end{itemize}
\end{lemma}
\begin{proof}
    Take a complex-linear splitting $\rho:T\to W$ of the short exact sequence and a real basis $\{\wt{e_1},\wt{e_2}=j\wt{e_1}\}$ of $T$. Define $e_1=\rho(\wt{e_1})$ and $e_2=\rho(\wt{e_2})$ and $\omega_1=\Omega(e_2,\cdot,\cdot)|_V$ and $\omega_2=-\Omega(e_1,\cdot,\cdot)|_V$. Then
    \begin{align*}
        \omega_1(\cdot,I\cdot) &= \Omega(e_2,\cdot,I\cdot)|_V = -\Omega(Ie_2,\cdot,\cdot)|_V = -\omega_2.
    \end{align*}
    So $\omega_1$ and $\omega_2$ define a linear CRPS structure on $(V,I)$. By the Darboux theorem for linear CRPS forms, there exists a basis $\{a_1^k,a_2^k,b_1^k,b_2^k\}_{k=1}^n$ of $V$ such that $Ia_1^k=a_2^k$, $Ib_1^k=-b_2^k$ and 
    \begin{align*}
        \omega_1 = \sum_{k=1}^n \left(\beta_1^k\wedge \alpha_1^k+\beta_2^k\wedge \alpha_2^k\right), && \omega_2 = \sum_{k=1}^n\left(\beta_1^k\wedge\alpha_2^k-\beta_2^k\wedge \alpha_1^k\right).
    \end{align*}
    This fixes the action of $\Omega$ on at least 2 vertical vectors. The only other term that can appear is of the form $\nu\wedge \epsilon_1\wedge\epsilon_2$, by the dimension of $T$. 
\end{proof}

\section{Proof of the Darboux theorem}\label{sec:Darboux}
The proof is inspired by the proof of the Darboux theorem for holomorphic symplectic manifolds from \cite{wagner2023pseudo}. We start by taking a local trivialization of $\sigma$, that we shrink in order to get coordinate charts. In other words, we have opens $\mathbb{U}\subseteq W$ and $\mathbb{S}=\sigma(\mathbb{U})\subseteq\Sigma$ and a commutative diagram
% https://q.uiver.app/#q=WzAsNCxbMCwwLCJcXG1hdGhiYntVfSJdLFswLDEsIlxcbWF0aGJie1N9Il0sWzEsMCwiXFxoYXR7U31cXHRpbWVzXFxoYXR7VX0iXSxbMSwxLCJcXGhhdHtTfSJdLFswLDEsIlxccGkiLDIseyJvZmZzZXQiOi00fV0sWzAsMiwiXFxoYXR7XFxwc2l9Il0sWzEsMywiXFxoYXR7XFxjaGl9IiwyXSxbMiwzLCJcXGhhdHtcXHByfSJdXQ==
\[\begin{tikzcd}
	{\mathbb{U}} & {\hat{S}\times\hat{U}} \\
	{\mathbb{S}} & {\hat{S}}
	\arrow["{\hat{\psi}}", from=1-1, to=1-2]
	\arrow["\sigma"', from=1-1, to=2-1]
	\arrow["{\hat{\pr}}", from=1-2, to=2-2]
	\arrow["{\hat{\chi}}"', from=2-1, to=2-2]
\end{tikzcd}\]
where $\hat{S}\subseteq\R^2$ and $\hat{U}\subseteq\R^{4n}$. Here, both $\hat{\psi}$ and $\hat{\chi}$ are diffeomorphisms. We put almost complex structures $\hat{I}$ on $\hat{S}\times\hat{U}$ and $\hat{j}$ on $\hat{S}$ by 
\begin{align*}
    \hat{I}&=d\hat{\psi}\circ I\circ d\hat{\psi}^{-1}\\
    \hat{j}&=d\hat{\chi}\circ j\circ d\hat{\chi}^{-1}.
\end{align*}
As $\hat{\pr}$ is holomorphic with respect to these almost complex structures, we can write for $(t,u)\in \hat{S}\times\hat{U}$
\begin{align*}
    \hat{I}_{(t,u)}=\begin{pmatrix}
        \hat{j}_{t} & 0\\ \hat{A}_{(t,u)} & \hat{J}_{(t,u)}
    \end{pmatrix},
\end{align*}
for some complex matrix $\hat{J}_{(t,u)}$ on $T_u\hat{U}$ and a map $\hat{A}_{(t,u)}:T_t\hat{S}\to T_u\hat{U}$ satisfying $\hat{A}_{(t,u)}\hat{j}_t+\hat{J}_{(t,u)}\hat{A}_{(t,u)}=0$. In other words, we get a $t$-dependent almost complex structure $\hat{J}_t$ on $\hat{U}$. 

Now, since $I$ and $j$ are integrable, so are $\hat{I}$ and $\hat{j}$. For vector fields $Y_1$ and $Y_2$ on $\hat{U}$ we have 
\begin{align*}
    0=N_{\hat{I}}((0, Y_1),(0,Y_2))=(0,N_{\hat{J}_t}(Y_1,Y_2)).
\end{align*}
So $\hat{J}_t$ is integrable for every $t$. Therefore, after possibly shrinking $\hat{U}$ and $\hat{S}$, there exists a biholomorphism $\chi':\hat{S}\to S\subseteq\C$ and a smooth family of biholomorphisms $\nu_t':\hat{U}\to U\subseteq\C^{2n}$, where $\C$ and $\C^{2n}$ are equipped with the standard complex structure. We can combine them into a map $\psi':\hat{S}\times\hat{U}\to S\times U$ defined by $\psi'(t,u)=(\chi'(t),\nu_t'(u))$, yielding the following commutative diagram
% https://q.uiver.app/#q=WzAsNixbMCwwLCJcXG1hdGhiYntVfSJdLFswLDEsIlxcbWF0aGJie1N9Il0sWzEsMCwiXFxoYXR7U31cXHRpbWVzXFxoYXR7VX0iXSxbMSwxLCJcXGhhdHtTfSJdLFsyLDAsIlNcXHRpbWVzIFUiXSxbMiwxLCJTIl0sWzAsMSwiXFxwaSIsMix7Im9mZnNldCI6LTR9XSxbMCwyLCJcXGhhdHtcXHBzaX0iXSxbMSwzLCJcXGhhdHtcXGNoaX0iLDJdLFsyLDMsIlxcaGF0e1xccHJ9Il0sWzIsNCwiXFxwc2knIl0sWzMsNSwiXFxjaGknIl0sWzQsNSwiXFxwciIsMV1d
\[\begin{tikzcd}
	{\mathbb{U}} & {\hat{S}\times\hat{U}} & {S\times U} \\
	{\mathbb{S}} & {\hat{S}} & S
	\arrow["{\hat{\psi}}", from=1-1, to=1-2]
	\arrow["\sigma"', shift left=4, from=1-1, to=2-1]
	\arrow["{\psi'}", from=1-2, to=1-3]
	\arrow["{\hat{\pr}}", from=1-2, to=2-2]
	\arrow["\pr", from=1-3, to=2-3]
	\arrow["{\hat{\chi}}"', from=2-1, to=2-2]
	\arrow["{\chi'}"', from=2-2, to=2-3]
\end{tikzcd}\]
where the top horizontal arrows are diffeomorphisms and the bottom horizontal arrows are biholomorphisms. Note that $d\psi'\circ \hat{I}|_{\hat{V}}=i\circ d\psi'|_{\hat{V}}$, for $\hat{V}=\ker d\hat{\pr}$.

We define $\Omega_0=((\psi'\circ\hat{\psi})^{-1})^*\Omega$ on $S\times U$. By \cref{lem:linDarboux}, we may assume that the action of $\Omega_0|_{w'}$ on at least two vertical vectors is standard, where $w'=\psi'\circ\hat{\psi}(w)$. First, we need to check that $\Omega_0$ is CRMS. This follows from the following lemma.
\begin{lemma}
    Let $\sigma:W\to\Sigma$ and $\sigma':W'\to\Sigma'$ two fiber bundles, where $(W,I)$ and $(W',I')$ are complex manifolds, $(\Sigma,j)$ and $(\Sigma',j')$ are Riemann surfaces, and $\sigma$ and $\sigma'$ are holomorphic. Let $\psi:W'\to W$ a diffeomorphism covering a biholomorphism $\chi:\Sigma'\to\Sigma$ such that $d\psi\circ I'|_V=I\circ d\psi|_V$. If $\Omega$ is a CRMS form on $W$ and $\Omega'=\psi^*\Omega$, then $\Omega'$ is CRMS on $W'$.
\end{lemma}
\begin{proof}
    The closedness condition follows trivially. Also, note that for $v\in V'$ we have $d\sigma(d\psi(v))=d\chi(d\sigma'(v))=0$, so that $d\psi$ maps $V'$ into $V$. This shows that $\Omega'$ satisfies conditions \ref{1hor} and \ref{fiberwisenondeg} of \cref{def:CRMSbundle}. Finally, we must check that $\Omega'$ is compatible with $I'$ in the sense of condition \ref{Icomp}. 

    First, note that for $\xi\in T_{w'}W$
    \begin{align*}
        d\sigma(d\psi(I'\xi))=d\chi\circ d\sigma'(I'\xi) = j\circ d\chi\circ d\sigma'(\xi)=j\circ d\sigma\circ d\psi(\xi) = d\sigma(I d\psi(\xi)).
    \end{align*}
    So by \cref{rem:CRMSbundle}.\ref{vertindep} it follows that 
    \[\Omega(d\psi(I'\xi),\cdot,\cdot)|_V=\Omega(Id\psi(\xi),\cdot,\cdot)|_V.\]
    Therefore, for $v_1,v_2\in V'$,
    \begin{align*}
        \Omega'(I'\xi,v_1,v_2) &= \Omega(d\psi(I'\xi),d\psi(v_1),d\psi(v_2))\\
        &=\Omega(Id\psi(\xi),d\psi(v_1),d\psi(v_2))\\
        &=-\Omega(d\psi(\xi),d\psi(v_1),Id\psi(v_2))\\
        &=-\Omega(d\psi(\xi),d\psi(v_1),d\psi(I'v_2))\\
        &=-\Omega'(\xi,v_1,I'v_2).
    \end{align*}
    Here, the fourth equality follows from the assumption that $d\psi$ is complex-linear when applied to vertical vectors. 
\end{proof}

We define $\Omega_1$ to be the standard CRMS form in coordinates on $S\times U$ and $\Omega_r=\Omega_0+r(\Omega_1-\Omega_0)=:\Omega_0+r\eta$ for $r\in[0,1]$. Since the actions of $\Omega_0|_{w'}$ and $\Omega_1|_{w'}$ on at least two vertical vectors coincide, we may assume that $\Omega_r$ is fiberwise non-degenerate for any $r\in[0,1]$ in the sense of condition \ref{fiberwisenondeg} of \cref{def:CRMSbundle}, after possibly shrinking $S$ and $U$ again. Thus, $\Omega_r$ is CRMS on $S\times U$ for any $r\in[0,1]$.

After possibly shrinking again, we may assume $U$ is holomorphically contractible. That is, there exists a family of holomorphic maps $\rho_s':U\to U$ with $\rho_1'=\Id_U$, $\rho_0(u)=w'$ for all $u\in U$ and $\rho'_s(w')=w'$ for all $s$.  
Let $X_s'$ be the associated vector field on $U$ and $X_s=(0,X_s')$ the vertical lift to $S\times U$. Also, let $\rho_s=\Id_S\times\rho_s':S\times U\to S\times U$. We define a 2-form on $S\times U$ by
\[\alpha = \int_0^1\rho_s^*(\iota(X_s)\eta)\,ds.\]
Then 
\[d\alpha = \int_0^1\frac{d}{ds}(\rho_s^*\eta)\,ds=\eta-\rho_0^*\eta,\]
so that $d\alpha(\xi,v_1,v_2)=\eta(\xi,v_1,v_2)$ for any vertical vectors $v_1,v_2$, as $d\rho_0(v)=0$ for any vertical $v$. Also
\begin{align*}
    \alpha(i\xi,v) &= \int_0^1\eta(X_s,d\rho_s(i\xi),d\rho_s(v))\,ds\\
    &=\int_0^1 \eta(X_s,id\rho_s(\xi),d\rho_s(v))\,ds\\
    &=-\int_0^1\eta(id\rho_s(\xi),X_s,d\rho_s(v))\,ds\\
    &=\int_0^1\eta(d\rho_s(\xi),X_s,id\rho_s(v))\,ds\\
    &=-\int_0^1\eta(X_s,d\rho_s(\xi),d\rho_s(iv))\,ds\\
    &=-\alpha(\xi,iv).
\end{align*}
Moreover, $\alpha$ is 1-horizontal, meaning that $\alpha(v_1,v_2)=0$, for any vertical vectors $v_1,v_2$. Without loss of generality, we may assume $\alpha|_{w'}=0$. %\textcolor{red}{(check this)}

Now, since $\Omega_r$ is fiberwise non-degenerate, we may define a vertical vector field $\V_r$ on $S\times U$, by 
\[\Omega_r(\del_1,\V_r,v)=\alpha(\del_1,v)\]
for any vertical $v$. Then
\begin{align*}
    \Omega_r(\del_2,\V_r,v) &= -\Omega_r(\del_1,\V_r,iv)\\
    &=-\alpha(\del_1,iv)\\
    &=\alpha(\del_2,v),
\end{align*}
since $\del_2=i\del_1$. So then $\Omega_r(\xi,\V_r,v)=\alpha(\xi,v)$ for any vector field $\xi$ on $S\times U$, as both sides of the equation only depend on $d\pr(\xi)$.

We want to show that $\L_{\V_r}\Omega_r(\xi,v_1,v_2)=-d\alpha(\xi,v_1,v_2)$ for any $\xi$ and any vertical $v_1$ and $v_2$. We write out both sides. The left-hand side equals
\begin{align*}
    \L_{\V_r}\Omega_r(\xi,v_1,v_2) &= d(\iota(\V_r)\Omega_r)(\xi,v_1,v_2)\\
    &= \xi\cdot \Omega_r(\V_r,v_1,v_2)-v_1\cdot\Omega_r(\V_r,\xi,v_2)+v_2\cdot\Omega_r(\V_r,\xi,v_1)\\
    &\phantom{bla}-\Omega_r(\V_r,[\xi,v_1],v_2)+\Omega_r(\V_r,[\xi,v_2],v_1)-\Omega_r(\V_r,[v_1,v_2],\xi)\\
    &=v_1\cdot\alpha(\xi,v_2)-v_2\cdot\alpha(\xi,v_1)\\
    &\phantom{bla}+\alpha([\xi,v_1],v_2)-\alpha([\xi,v_2],v_1)-\alpha(\xi,[v_1,v_2]).
\end{align*}
Similarly, minus the right-hand side equals
\begin{align*}
    d\alpha(\xi,v_1,v_2) &= \xi\cdot \alpha(v_1,v_2) -v_1\cdot\alpha(\xi,v_2)+v_2\cdot\alpha(\xi,v_1)\\
    &\phantom{bla} -\alpha([\xi,v_1],v_2)+\alpha([\xi,v_2],v_1)-\alpha([v_1,v_2],\xi)\\
    &=-v_1\cdot\alpha(\xi,v_2)+v_2\cdot\alpha(\xi,v_1)\\
    &\phantom{bla} -\alpha([\xi,v_1],v_2)+\alpha([\xi,v_2],v_1)+\alpha(\xi,[v_1,v_2]).
\end{align*}
Here, the last equality follows from the fact that $\alpha$ is 1-horizontal. So indeed, $\L_{\V_r}\Omega_r(\xi,v_1v_2)=-d\alpha(\xi,v_1,v_2)=-\eta(\xi,v_1,v_2)=-\frac{d}{dr}\Omega_r(\xi,v_1,v_2)$. 

Let $\vphi_r$ be the flow of $\V_r$. Note that $\pr\circ\vphi_r=\pr$ for all $r$, as $\V_r$ is vertical. Thus, $\vphi_r$ preserves the vertical bundle. We get that
\begin{align*}
    \frac{d}{dr}\left(\vphi_r^*\Omega_r(\xi,v_1,v_2)\right) = \vphi_r^*\left(\L_{\V_r}\Omega_r+\frac{d}{dr}\Omega_r\right)(\xi,v_1,v_2)=0.
\end{align*}
Therefore, $\vphi_1^*\Omega_1(\xi,v_1,v_2)=\Omega_0(\xi,v_1,v_2)$ for all vertical $v_1,v_2$. The following commuting diagram gives an overview of the maps we introduced.
% https://q.uiver.app/#q=WzAsOCxbMCwwLCJcXG1hdGhiYntVfSJdLFswLDEsIlxcbWF0aGJie1N9Il0sWzEsMCwiXFxoYXR7U31cXHRpbWVzXFxoYXR7VX0iXSxbMSwxLCJcXGhhdHtTfSJdLFsyLDAsIlNcXHRpbWVzIFUiXSxbMiwxLCJTIl0sWzMsMSwiUyJdLFszLDAsIlNcXHRpbWVzIFUiXSxbMCwxLCJcXHBpIiwyLHsib2Zmc2V0IjotNH1dLFswLDIsIlxcaGF0e1xccHNpfSJdLFsxLDMsIlxcaGF0e1xcY2hpfSIsMl0sWzIsMywiXFxoYXR7XFxwcn0iXSxbMiw0LCJcXHBzaSciXSxbMyw1LCJcXGNoaSciLDJdLFs0LDUsIlxccHIiXSxbNSw2XSxbNCw3LCJcXHZwaGlfMSJdLFs3LDYsIlxccHIiXV0=
\[\begin{tikzcd}
	{\mathbb{U}} & {\hat{S}\times\hat{U}} & {S\times U} & {S\times U} \\
	{\mathbb{S}} & {\hat{S}} & S & S
	\arrow["{\hat{\psi}}", from=1-1, to=1-2]
	\arrow["\sigma"', shift left=4, from=1-1, to=2-1]
	\arrow["{\psi'}", from=1-2, to=1-3]
	\arrow["{\hat{\pr}}", from=1-2, to=2-2]
	\arrow["{\vphi_1}", from=1-3, to=1-4]
	\arrow["\pr", from=1-3, to=2-3]
	\arrow["\pr", from=1-4, to=2-4]
	\arrow["{\hat{\chi}}"', from=2-1, to=2-2]
	\arrow["{\chi'}"', from=2-2, to=2-3]
	\arrow[from=2-3, to=2-4]
\end{tikzcd}\]

Let $\psi=\vphi_1\circ\psi'\circ\hat{\psi}$ and $\chi=\chi'\circ\hat{\chi}$. Then, $(\psi^{-1})^*\Omega(\xi,v_1,v_2)=\Omega_1(\xi,v_1,v_2)$ for all vertical $v_1,v_2$. Define \[\nu(\xi)=(\psi^{-1})^*\Omega(\xi,\del_1,\del_2)-\Omega_1(\xi,\del_1,\del_2)=(\psi^{-1})^*\Omega(\xi,\del_1,\del_2),\]
since $\Omega_1$ vanishes on two horizontal vectors. A priori, $\nu$ is a 1-form on $S\times U$. However, since $S$ is 2-dimensional, $\nu(\xi)$ only depends on the vertical part of $\xi$. Therefore, we may view $\nu_t$ as a $t$-dependent 1-form on $U$. We see that \[(\psi^{-1})^*\Omega=\Omega_1+\nu_t\wedge dt_1\wedge dt_2.\]
Moreover, $\nu_t$ is closed for each $t$, so that $\nu_t=-dH_t$, for some function $H_t:U\to\R$. 

The only thing left to prove is that $d\psi$ is complex-linear, when applied to vertical vectors. Note that we already proved that this was true for $\psi'\circ\hat{\psi}$. So we only have to show it for $\vphi_1$. The statement follows from the following \cref{lem:verticalholom} below. \qed

\begin{lemma}\label{lem:verticalholom}
    Let $\sigma:W\to\Sigma$ and $\sigma':W'\to\Sigma$ be two CRMS bundles over the same base, with CRMS forms $\Omega$ and $\Omega'$ respectively. If $\vphi:W\to W'$ is a diffeomorphism covering the identity on $\Sigma$, such that $\vphi^*\Omega'(\xi,v_1,v_2)=\Omega(\xi,v_1,v_2)$ for all $\xi\in T_wW$ and $v_1,v_2\in V_w$, then $d\vphi(Iv)=I'd\vphi(v)$ for all $v\in V_w$. 
\end{lemma}
\begin{proof}
    Let $\xi\in T_wW$ be any vector with $d\sigma(\xi)\neq 0$. Note that $d\sigma'\circ d\vphi(\xi)=d\sigma(\xi)\neq 0$.  For $v_1,v_2\in V_w$ we calculate
    \begin{align*}
        \Omega'(d\vphi(\xi),d\vphi(v_1),d\vphi(Iv_2)) &= \Omega(\xi,v_1,Iv_2)\\
        &= -\Omega(I\xi,v_1,v_2)\\
        &=-\Omega'(d\vphi(I\xi),d\vphi(v_1),d\vphi(v_2)).
    \end{align*}
    Now, note that $d\sigma'(d\vphi(I\xi))=d\sigma(I\xi)=j d\sigma(\xi)=j d\sigma'\circ d\vphi(\xi)=d\sigma'(I'd\vphi(\xi))$. So again by \cref{rem:CRMSbundle}.\ref{vertindep} we have 
    \begin{align*}
        -\Omega'(d\vphi(I\xi),d\vphi(v_1),d\vphi(v_2)) &= -\Omega'(I'd\vphi(\xi),d\vphi(v_1),d\vphi(v_2))\\
        &=\Omega'(d\vphi(\xi),d\vphi(v_1),I'd\vphi(v_2)).
    \end{align*}
    Thus $\Omega'(d\vphi(\xi),d\vphi(v_1),d\vphi(Iv_2))=\Omega'(d\vphi(\xi),d\vphi(v_1),I'd\vphi(v_2))$. Now, as $d\sigma'(d\vphi(\xi))\neq 0$, by the fiberwise non-degeneracy of $\Omega'$ and the bijectivity of $d\vphi$, we have that $d\vphi(Iv_2)=I'd\vphi(v_2)$ for all $v_2\in V_w$. 
\end{proof}

\section{Floer sections}\label{sec:Floer}
Now that we have defined the class of manifolds for which the Hamiltonian equations locally look like the ones that we find in CRPS geometry, we want to setup the stage for defining Floer theory. In order to do that we have to define Floer curves as the gradient of the action functional with respect to a suitable metric. Just like in symplectic geometry, the metric has to somehow relate the multisymplectic form, to almost complex structures. For a CRMS form $\Omega$ on $W$ and $\xi\in T_wW$, we define $\omega^\xi:=\Omega(\xi,\cdot,\cdot)|_{V_w}$. By definition of CRMS forms, this is a linear symplectic form on $V_w$, whenever $\xi$ is not purely vertical. 

% Next, we want to determine the existence of a metric that somehow relates the multisymplectic form to almost complex structures. Note that $\Lambda^C(Y)\to\wt{\Lambda^C}(Y)$ is a trivial bundle. We can pull back $\Theta$ by the zero section to get $\Theta^0$ on $\wt{\Lambda^C}(Y)$, locally given by
% \[\Theta^0=\sum_a\left((P_1^adq_1^a+P_2^adq_2^a)\wedge\lambda_\alpha(z)dz_2-(P_1^adq_2^a-P_2^adq_1^a)\wedge\lambda_\alpha(z)dz_1\right).\]
% For $\xi\in T\Sigma$, we define $\Omega_\xi:=-\Omega^0(\cdot,\cdot,\xi)$. Thus, if in local coordinates $\xi=\frac{\xi_1}{\sqrt{\lambda_\alpha(z)}}\del_{z_1}+\frac{\xi_2}{\sqrt{\lambda_\alpha(z)}}\del_{z_2}$, we get that 
% \[\Omega_\xi=\sqrt{\lambda_\alpha(z)}\sum_a\left(\xi_2(dP_1^a\wedge dq_1^a+dP_2^a\wedge dq_2^a)-\xi_1(dP_1^a\wedge dq_2^a-dP_2^a\wedge dq_1^a)\right)\]
% We also define the vertical bundle $V:=\ker(d\sigma^C)$.

\begin{proposition}\label{prop:globalmetric}
    Let $\sigma:(W,I,\Omega)\to(\Sigma,j)$ be a CRMS bundle and fix a metric $h$ on $\Sigma$ within the conformal class of $j$. There exists an inner product $g$ on the fibers of V and a bundle map $\J:\sigma^*T\Sigma\backslash\{0\}\to \text{End}(V)$, such that for all $\xi\in T_wW$ with $\rho=d\sigma(\xi)\neq0$, we have
    \begin{enumerate}[label=(\roman*)]
        \item\label{it:compl} $\J(\rho)^2=-|\rho|^2_h\Id_{V_w}$,
        \item\label{it:metric} $\omega^\xi = g(\cdot,\J(\rho)\cdot)$,
        \item\label{it:compatible} $\J(j\rho)=I \J(\rho)$.
    \end{enumerate}
\end{proposition}
\begin{remark}
    In the CRPS framework, the analogue of $\J$ would be $Jdt_2-Kdt_1$. Conditions \ref{it:compl} and\ref{it:metric} then translate to the fact that any linear combination of unit norm of $J$ and $K$ is an almost complex structure compatible with the corresponding linear combination of $\omega_1$ and $\omega_2$. Condition \ref{it:compatible} says that $K=IJ$.
\end{remark}

\begin{proof}[Proof of \cref{prop:globalmetric}]
    The proof is an adaptation of proposition 7.1 from \cite{from1to2}. Just like there, the essence of the proposition lies in the local version. The global statement then follows from the fact that after an initial choice of inner product on V, the rest of the construction is canonical.

    Thus, we fix $w\in W$ and an inner product $(\cdot,\cdot)$ on the vector space $V_w$ compatible with $I$. All transposes will be taken with respect to this inner product. Also fix $\xi\in T_wW$ such that $\rho=d\sigma(\xi)\in T_{\sigma(w)}\Sigma$ is non-zero. %Note that $I\del_{q_1^a}=\del_{q_2^a}$ and $I\del_{P_1^a}=-\del_{P_2^a}$ in local coordinates. 
    By non-degeneracy of $\omega^\xi$ we get a skew-symmetric linear isomorphism $A_\rho:V_w\to V_w$, such that $\omega^\xi=|\rho|_h(\cdot,A_\rho\cdot)$, where $A$ only depends on $\rho$ by \cref{rem:CRMSbundle}.\ref{vertindep} In fact, $A_{c\rho}=A_\rho$ for any $c>0$ and $A_{-\rho}=-A_\rho$. Note that \[\omega^\xi(I v_1,Iv_2)=-\omega^{I\xi}(Iv_1,v_2)=\omega^{I\xi}(v_2,Iv_1)=\omega^\xi(v_2,v_1)=-\omega^\xi(v_1,v_2),\] so that $A_\rho$ anticommutes with $I$. 

    Let $B_\rho:=\sqrt{A_\rho A_\rho^T}$ and $J_\rho:=|\rho|_hB_\rho^{-1}A_\rho$. From the skew-symmetry of $A_\rho$, it follows that $J_\rho$ is skew-symmetric and $J_\rho J_\rho^T=|\rho|_h^2\Id_{V_w}$, so that indeed $J_\rho$ is an almost complex structure on $V_w$ for $\rho$ of unit length. %We define the inner product $g_\rho=(\cdot,B_\rho\cdot)$, so that indeed $\omega^\xi=g_\rho(\cdot, J_\rho\cdot)$. 
    We want to show that $B_\rho$ is independent of $\rho$. 

    % \textcolor{red}{rewrite}
    % To this end, define $e_l=\lambda_\alpha(z)^{-\frac{1}{2}}\del_{z_l}$ for $l=1,2$. Note that $A_\xi=A_1\xi_1+A_2\xi_2$, where $A_1$ and $A_2$ are given by
    % \begin{align*}
    %     (\cdot,A_l\cdot)=\Omega_{e_l}.
    % \end{align*}
    % Since $\Omega_{e_1}=\Omega_{e_2}(\cdot,I\cdot)$, we get that $A_1=A_2I=-IA_2$. It follows that $A_1A_1^T=A_2II^TA_2^T=A_2A_2^T$ and $A_1A_2^T=-A_2A_1^T$, which yields that 
    % \[A_\xi A_\xi^T=(\xi_1^2+\xi_2^2)A_1A_1^T=A_1A_1^T\]
    % for all $\xi$ of unit norm. This proves that $B_\xi$ is independent of the choice of $\xi$. So $g=g_\xi$ is a well-defined inner product. 

    Note that $\omega^{I\xi}=-\omega^\xi(\cdot,I\cdot)$. The left-hand side is equal to $|\rho|_h(\cdot,A_{j\rho}\cdot)$ and the right-hand side to $-|\rho|_h(\cdot,A_\rho I\cdot)$. This yields that $A_{j\rho}=-A_\rho I$. From this we conclude that 
    \[B_{j\rho}^2=A_{j\rho}A_{j\rho}^T=A_\rho II^TA_\rho^T=A_\rho A_\rho^T=B_\rho^2.\]
    Also $B_\rho=B_{-\rho}$. Since $\rho$ and $j\rho$ span $T_{\sigma(w)}\Sigma$, we get that $B_\rho$ is independent of $\rho$ (as long as $\rho$ is non-zero). Thus we may define the inner product $g=(\cdot,B\cdot)$ independent of $\rho$. Indeed, we see that $\omega^\xi=g(\cdot, J_\rho\cdot)$.

    % To prove property \ref{it:compatible}, note that $A_{j\xi}=-\xi_2A_1+\xi_1A_2$, which yields $J_{j\xi}=-\xi_2J_{e_1}+\xi_1J_{e_2}$. Also, 
    % \[J_{e_1}I=B^{-1}A_1I=-B^{-1}A_2=-J_{e_2}\]
    % and 
    % \[-\Id_{V_\eta}=J_{e_2}^2=(J_{e_1}I)^2=J_{e_1}IJ_{e_1}I,\]
    % so that $IJ_{e_1}=-J_{e_1}I=J_{e_2}$. Thus $IJ_\xi=\xi_1J_{e_2}-\xi_2J_{e_1}=J_{j\xi}$.

    Finally, note that
    \[J_{j\rho}=|j\rho|_hB^{-1}A_{j\rho}=-|\rho|_hB^{-1}A_\rho I=|\rho|_hIB^{-1}A_\rho=IJ_\rho.\]

    We conclude that the map $\J:\rho\mapsto J_\rho$ satisfies the conditions of the theorem.   
\end{proof}

There is a canonically defined action functional on sections of a CRMS bundle $\sigma:W\to\Sigma$. If we assume $\Omega=d\Theta$, then it is given by
\[\A(Z)=\int_\Sigma Z^*\Theta,\]
for a section $Z:\Sigma\to W$. Note that the tangent space to $\Gamma(\sigma)$ at $Z$ can be seen as $\Gamma(Z^*V)$. We compute for $Z\in\Gamma(\sigma)$ and $\dot{Z}\in\Gamma(Z^*V)$
\begin{align}\label{eq:actionderivative}
    d\A(Z)(\dot{Z}) &=\int_\Sigma Z^*(\iota(\dot{Z})\Omega).
\end{align}
Fix a metric $h$ on $\Sigma$, an inner product $g$ on $V$ as in \cref{prop:globalmetric} and an Ehresmann connection on $\sigma$. Then we denote for $\rho\in T\Sigma$ the vertical projection of $dZ(\rho)$ by $\nabla_\rho Z$. The metrics induce an $L^2$-metric on $\Gamma(Z^*V)$ by $\langle\cdot,\cdot\rangle=\int_\Sigma g(\cdot,\cdot)\,d\V_h$, where $d\V_h$ is the volume form of $h$ on $\Sigma$.

Now, in a positive local orthonormal frame $\{\rho_1,\rho_2=j\rho_1\}$ of $T\Sigma$, the integrand in \cref{eq:actionderivative} is given by 
\begin{align*}
    \Omega(\dot{Z},dZ(\rho_1),dZ(\rho_2))d\V_h,
\end{align*}
and 
\begin{align*}
    \Omega(\dot{Z},dZ(\rho_1),dZ(\rho_2)) &= \omega^{\rho_2}(\dot{Z},\nabla_{\rho_1}Z)-\omega^{\rho_1}(\dot{Z},\nabla_{\rho_2}Z)+\Omega(\dot{Z},\rho_1,\rho_2)\\
    &=g\left(\dot{Z},\J(\rho_2)\nabla_{\rho_1}Z-\J(\rho_1)\nabla_{\rho_2}Z\right)-g(\dot{Z},\alpha(Z)).
\end{align*}
Here $\alpha(Z)\in V_{Z}$ is defined by $g(\dot{Z},\alpha(Z))=-\Omega(\dot{Z},\rho_1,\rho_2)$ for $\dot{Z}$ vertical. Note that $\alpha(Z)$ does not depend on the derivatives of $Z$. 
We denote 
\begin{align}\label{eq:Jnabla}
\J_\nabla Z(\rho):=\J(\rho)\left(-I\nabla_\rho-\nabla_{j\rho}\right)Z.
\end{align}
Thus, the gradient of $\A(Z)$ is locally given by $\J_\nabla Z(\rho_1)-\alpha(Z)$. We see that
\begin{align*}
    \J_\nabla Z(j\rho) =-\J(\rho)I(-I\nabla_{j\rho}+\nabla_\rho)Z=\J_\nabla Z(\rho),
\end{align*}
for any non-vanishing local vector field $\rho$. Note that $\{\rho,j\rho\}$ locally spans $T\Sigma$, so that $\J_\nabla Z(\rho)$ is only dependent on $|\rho|_h^2$. Thus, we may define $\J_\nabla Z:=\J_\nabla Z(\rho_1)$, which is independent of the chosen local orthonormal frame. We see that $\J_\nabla$ is actually a globally well-defined operator on $\Gamma(\sigma)$, which is locally given by \cref{eq:Jnabla} for any unit vector field $\rho$. Then globally, the gradient of $\A$ is given by 
\[
\grad\A(Z) = \J_\nabla Z-\alpha(Z),
\]
where the zero-order term $\alpha(Z)$ has to be globally well-defined, since all other terms in the equation are. 

As always, we may define Floer curves to be the negative gradient lines of $\A$. Denote $\pr_{\Sigma}:\R\times\Sigma\to\Sigma$ to be the projection. Then, Floer curves are given by sections of $\pr_\Sigma^* W$ satisfying
\begin{align*}
    \nabla_{\del_s}Z+\J_\nabla Z=\alpha(Z),
\end{align*}
or equivalently 
\begin{align}\label{eq:FueterFloer}
    I\nabla_{\del_s}Z+I\J_\nabla Z=I\alpha(Z).
\end{align}
We show that this is a perturbed Fueter equation. Let $M=\R\times\Sigma$ with metric $ds^2+h$, for $s$ the coordinate on $\R$. For $(\rho_s,\rho)\in STM$, we define $\wt{I}(\rho_s,\rho):=\rho_sI+\J(\rho)$. Then an easy check shows $\wt{I}(\rho_s,\rho)$ is an almost complex structure on the fiber of $\pr^*_\Sigma W$ and that \cref{eq:FueterFloer} is given by 
\[\sum_{l=1}^3\wt{I}(v_l)\nabla_{v_l}Z=I\alpha(Z),\]
in any local orthonormal frame $\{v_1,v_2,v_3\}=\{\del_s,\rho,j\rho\}$ of $M$. This is the Fueter equation as described in \cite{walpuski2017compactness} and is an elliptic equation. 

The discussion above allows us to employ the elliptic methods that made symplectic geometry flourish since the 1980's, in the setting of multisymplectic geometry. Indeed, in symplectic geometry pseudo-holomorphic curves and Floer theory are the main tools to prove results about symplectic manifolds and Hamiltonian flows on them.

%The only difference with the setting of \cite{walpuski2017compactness} is that in our case the almost complex structures $\J(\rho)$ and $\J(j\rho)$ are not necessarily integrable. 

In \cite{walpuski2017compactness}, the compactness theory for solutions to the Fueter equation is worked out and should guide the development of Floer theory for sections of $\sigma$, at least in the case where $\J(\rho)$ and $\J(j\rho)$ are integrable.

%The only difference with the setting of \cite{walpuski2017compactness} is that in our case the almost complex structures $\J(\rho)$ and $\J(j\rho)$ are not necessarily integrable. It should be noted that in the case where $\J(\rho)$ and $\J(j\rho)$ are not integrable, the proof of the monotonicity formula (Proposition 2.1 in \cite{walpuski2017compactness}) does not hold, so that more work has to be done to conclude compactness. 

%Nonetheless, \cref{eq:FueterFloer} can be used to prove the existence of solutions of the multisymplectic equation $Z^*(X\lrcorner\Omega)=0$ for all $X\in\mathfrak{X}(W)$. 
This line of reasoning was used in \cite{from1to2} to prove a cuplength result for perturbed harmonic maps $\T^2\to\T^{2n}$, when $\Sigma=\T^2$ and $W=\Sigma\times T^*\T^{2n}$. Similar arguments should work for $W=\wt{\Lambda^C(Y)}$, where $Y=\T^{2n}\times\Sigma$ and $\Sigma$ is a general Riemann surface. We decided to leave the formal proof out of this paper in order to focus the reader's attention on the additional geometric structures that arise when moving away from the torus to arbitrary Riemann surfaces.

\bibliography{mybib}{}
\bibliographystyle{alpha}

\end{document}